\documentclass{amsart}
\usepackage{amssymb,color,hhline}
\usepackage{amsmath}
\usepackage{a4}
\newtheorem{theorem}{Theorem}[section]
\newtheorem{lemma}[theorem]{Lemma}
\newtheorem{remark}[theorem]{Remark}

\makeatletter
  
  \@addtoreset{equation}{section}
 \makeatother

\def\mapright#1{\smash{\mathop{\longrightarrow}\limits\sp{#1}}}
\def\qedbox{\hbox{$\rlap{$\sqcap$}\sqcup$}}
\def\WW{\mathcal{W}}\def\WWW{{\tilde{\mathcal{W}}}}
\begin{document}
  \title{Geometric Realizations of para-Hermitian curvature models}
\author[M.  Brozos-V\'azquez et. al.]{M.  Brozos-V\'azquez, P. Gilkey, S. Nik\v cevi\'c, and  R.
V\'{a}zquez-Lorenzo}
\address{MB: Department of Mathematics, University of A Coru\~na, Spain\\
E-mail: mbrozos@udc.es}
\address{PG: Mathematics Department, University of Oregon\\
   Eugene OR 97403 USA\\
   E-mail: gilkey@uoregon.edu}
\address{SN: Mathematical Institute, Sanu,
Knez Mihailova 35, p.p. 367\\
11001 Belgrade,
Serbia\\
E-mail: stanan@mi.sanu.ac.rs}
\address{RV: Department of Geometry and Topology, Faculty of Mathematics,
University of Santiago de Compostela, Santiago de Compostela,
Spain\\
E-mail: ravazlor@edu.xunta.es}
\begin{abstract} We show that a para-Hermitian algebraic curvature model satisfies the para-Gray identity if
and only if it is geometrically realizable by a para-Hermitian manifold. This requires extending the Tricerri-Vanhecke curvature decomposition to the
para-Hermitian setting. Additionally, the geometric realization can be chosen to have constant scalar curvature and constant $\star$-scalar curvature.
\end{abstract}
\maketitle

\centerline{This paper is dedicated to the memory of \bf Professor Katsumi Nomizu}

\section{Introduction}

\subsection{Hermitian geometry}
Let $g$ be a Riemannian metric on a smooth manifold $M$ of dimension $2n$. Let $\mathcal{J}$ give $(M,g)$ an {\it almost Hermitian} structure. This means that
$\mathcal{J}$ is an almost complex structure on the tangent bundle which is compatible with $g$, i.e.
$\mathcal{J}^2=-\operatorname{id}$ and $\mathcal{J}^*g=g$. We say that the {\it almost Hermitian} manifold
$$\mathcal{M}:=(M,g,\mathcal{J})$$ is  {\it Hermitian} if $\mathcal{J}$ is integrable, i.e. if the Nijenhuis tensor vanishes or, equivalently, there exist local
coordinates
$(x_1,...,x_n,y_1,...,y_n)$  centered at any given point of the manifold so that
$$\mathcal{J}\partial_{x_i}=\partial_{y_i}\quad\text{and}\quad\mathcal{J}\partial_{y_i}=-\partial_{x_i}\,.$$
We refer to \cite{KN69} for further details.

The Riemann curvature tensor
$$ R(x,y):=\nabla_x\nabla_y-\nabla_y\nabla_x-\nabla_{[x,y]}$$
of the Levi-Civita
connection
\cite{N56} satisfies:
\begin{equation}\label{eqn-1.a}
\begin{array}{l}
R(x,y,z,w)+R(y,z,x,w)+R(z,x,y,w)=0,\\
R(x,y,z,w)=-R(y,x,z,w)=R(z,w,x,y)\,.\phantom{\vrule height 12pt}
\end{array}\end{equation}
Gray \cite{gray} showed that there is
an additional identity, which is called the Gray identity, which is satisfied by the curvature tensor of any Hermitian manifold:
\begin{eqnarray}
0&=&R(x,y,z,w)+R(Jx,Jy,Jz,Jw)-R(Jx,Jy,z,w)\nonumber\\
&-&R(Jx,y,Jz,w)-R(Jx,y,z,Jw)-R(x,Jy,Jz,w)\label{eqn-1.b}\\
&-&R(x,Jy,z,Jw)-R(x,y,Jz,Jw)\,.\nonumber
\end{eqnarray}

All  universal curvature symmetries for Hermitian manifolds are generated by the relations of Equations (\ref{eqn-1.a}) and
(\ref{eqn-1.b}). By contrast, there are no additional symmetries beyond those of Equation (\ref{eqn-1.a}) in the
almost Hermitian context. One can make this
statement precise as follows. Let $\langle\cdot,\cdot\rangle$ be a positive definite inner product on a real vector space $V$ of
dimension
$2n$. Let $J$ be a Hermitian complex structure on $V$; $J^2=-\operatorname{id}$ and
$J^*\langle\cdot,\cdot\rangle=\langle\cdot,\cdot\rangle$. Let
$A\in\otimes^4V^*$ be an {\it algebraic curvature tensor}, i.e. $A$ satisfies the symmetries of Equation (\ref{eqn-1.a}). Let
$$\mathfrak{C}:=(V,\langle\cdot,\cdot\rangle,J,A)$$
be the associated {\it Hermitian curvature model}.
We say that $\mathfrak{C}$ is geometrically realized by an almost Hermitian manifold $\mathcal{M}=(M,g,\mathcal{J})$ if there is an isomorphism $\phi:V\rightarrow
T_PM$ for some
$P\in M$ so that
$\phi^*g_P=\langle\cdot,\cdot\rangle$, $\phi^*\mathcal{J}_P=J$, and $\phi^*R_P=A$. We refer to \cite{GBKNW08,BGKS08} for the proof of the following result:
\begin{theorem}\label{thm-1.1}
Let $\mathfrak{C}$ be a Hermitian curvature model.
\begin{enumerate}
\item $\mathfrak{C}$ is always geometrically realized by an almost Hermitian manifold.
\item $\mathfrak{C}$ is geometrically realized by a Hermitian manifold if and only if $\mathfrak{C}$ satisfies {\rm Equation (\ref{eqn-1.b})}.
\end{enumerate}
\end{theorem}

There are analogous questions in the affine setting.  For example, if $\nabla $ is both holomorphic and affine Kaehler, then $R=0$ and $\nabla $ is locally flat
\cite{NP}.

\subsection{Para-Hermitian geometry} Let $(\tilde M,\tilde g)$ be a pseudo-Riemannian manifold of dimension $2n$. Let
$\tilde{\mathcal{J}}$ give
$(\tilde M,\tilde g)$ an {\it almost para-Hermitian} structure;
$\tilde{\mathcal{J}}^2=\operatorname{id}$ and
$\tilde{\mathcal{J}}^*\tilde g=-\tilde g$. In this setting, necessarily $\tilde g$
has neutral signature $(n,n)$. The {\it almost para-Hermitian} manifold
$$\tilde{\mathcal{M}}:=(\tilde M,\tilde g,\tilde{\mathcal J})$$
is said to be {\it para-Hermitian} if $\tilde{\mathcal{J}}$ is integrable, i.e.  if the Nijenhuis tensor
$N_{\tilde{\mathcal{J}}}$ vanishes (see, for instance, \cite{CMMS03}), where
$$N_{\tilde{\mathcal{J}}}(x,y):=[x,y]-\tilde{\mathcal{J}}[\tilde{\mathcal{J}}x,y]
-\tilde{\mathcal{J}}[x,\tilde{\mathcal{J}}y]+[\tilde{\mathcal{J}}x,\tilde{\mathcal{J}}y].$$
Equivalently, there exist local coordinates $(x_1,...,x_n,y_1,...,y_n)$
centered at any given point of $\tilde M$ so  that
$$\tilde{\mathcal{J}}\partial_{x_i}=\partial_{y_i}\quad\text{and}\quad\tilde{\mathcal{J}}\partial_{y_i}=\partial_{x_i}\,.$$
In the algebraic setting, let $\widetilde{\langle\cdot,\cdot\rangle}$ be a neutral signature inner product on a finite dimensional vector space
$\tilde V$. Let $\tilde J$ be a {\it para-Hermitian structure} on $(\tilde V,\widetilde{\langle\cdot,\cdot\rangle})$, i.e.
$\tilde J^2=\operatorname{id}$ and $\tilde J^*\widetilde{\langle\cdot,\cdot\rangle}=-\widetilde{\langle\cdot,\cdot\rangle}$. If $\tilde
A\in\otimes^4\tilde V^*$ is an algebraic curvature tensor, let
$$\tilde{\mathfrak{C}}:=(\tilde V,\widetilde{\langle\cdot,\cdot\rangle},\tilde J,\tilde A)$$
 be the corresponding {\it para-Hermitian curvature model}. We change the signs in Equation (\ref{eqn-1.b}) to define a corresponding {\it para-Gray} relation
\begin{eqnarray}
0&=&\tilde A(x,y,z,w)+\tilde A(\tilde Jx,\tilde Jy,\tilde Jz,\tilde Jw)+\tilde A(\tilde Jx,\tilde Jy,z,w)\nonumber\\
&+&\tilde A(\tilde Jx,y,\tilde Jz,w)+\tilde A(\tilde Jx,y,z,\tilde Jw)+\tilde A(x,\tilde Jy,\tilde Jz,w)\label{eqn-1.c}\\
&+&\tilde A(x,\tilde Jy,z,\tilde Jw)+\tilde A(x,y,\tilde Jz,\tilde Jw)\,.\nonumber
\end{eqnarray}

Assertion (1) in the following Theorem was established in \cite{GBKNW08}; Assertion (2) is the main new result of this paper:

\begin{theorem}\label{thm-1.2}
Let $\tilde{\mathfrak{C}}$ be a para-Hermitian curvature model.
\begin{enumerate}
\item $\tilde{\mathfrak{C}}$ is always geometrically realized by an almost para-Hermitian manifold.
\item $\tilde{\mathfrak{C}}$ is geometrically realized by a para-Hermitian manifold
if and only if $\tilde{\mathfrak{C}}$ satisfies {\rm Equation (\ref{eqn-1.c})}.
\end{enumerate}\end{theorem}

\begin{remark}\label{rmk-1.3}
\rm We make the following observations:
\begin{enumerate}
\item
The results of \cite{GBKNW08} show that the  manifolds in Theorems \ref{thm-1.1} and \ref{thm-1.2} can be
chosen to have constant scalar curvature and constant $\star$-scalar curvature.
\item The methods we will develop to establish Theorem \ref{thm-1.2}  (2) can be used to show that  Theorem
\ref{thm-1.1} holds for pseudo-Riemannian manifolds; it is not necessary to assume that the inner product
is positive definite.
\item In the Hermitian setting, let $\Omega(\cdot,\cdot):=g(\cdot,\mathcal{J}\cdot)$ be the Kaehler form;
in the para-Hermitian setting, the para-Kaehler
 form is defined similarly by setting
$\tilde\Omega(\cdot,\cdot):=\tilde g(\cdot,\tilde{\mathcal{J}}\cdot)$. The
geometric realizations can be chosen so that $d\Omega_P=0$ in the Hermitian setting or
$d\tilde{\Omega}_P=0$ in the para-Hermitian setting. Thus
requiring the Kaehler or the para-Kaehler
 identity (i.e. $d\Omega=0$ or $d\tilde\Omega=0$)  at a single point imposes no
additional curvature restrictions although, of course requiring the Kaehler identity globally yields additional
curvature restrictions.
\end{enumerate}\end{remark}

\subsection{Outline of the paper} Here is a brief outline to the paper. In Section
\ref{sect-2}, we will show that the curvature tensor of any para-Hermitian manifold satisfies Equation (\ref{eqn-1.c}) and thereby establish one implication of
Theorem
\ref{thm-1.2} (2). Rather than generalizing Gray's proof from the Hermitian to the para-Hermitian setting, we have chosen to give a direct proof which is quite
different in flavor. In Section \ref{sect-3}, we recall the Tricerri-Vanhecke
\cite{TV81} decomposition of the space of algebraic curvature tensors in the Hermitian setting and extend it to the para-Hermitian setting by complexification;
this result is perhaps of interest in its own right. In Section
\ref{sect-4}, we linearize the problem. We define a linear subspace $\mathfrak{P}$ of the space of all algebraic curvature tensors which is invariant under the
para-unitary structure group such that any element of $\mathfrak{P}$ can be realized by a para-Hermitian metric with vanishing Kaehler form at the point
in question. We complete the proof of Theorem \ref{thm-1.2} (2) in Section \ref{sect-5} by showing the elements of $\mathfrak{P}$ are precisely those algebraic
curvature tensors which satisfy the para-Gray identity given in Equation (\ref{eqn-1.c}).

\section{The para-Gray identity for para-Hermitian manifolds}\label{sect-2}
Let $\tilde J$ be a para-Hermitian structure on $(\tilde V,\widetilde{\langle\cdot,\cdot\rangle})$. Let $\{\tilde e_a\}$ be a basis for $\tilde V$. If
$\tilde T\in\otimes^4\tilde V^*$, we define the {\it para-Gray symmetrization}
\begin{eqnarray*}
\tilde{\mathcal{G}}(\tilde T)(\tilde e_a,\tilde e_b,\tilde e_c,\tilde e_d):&=&
\tilde T(\tilde e_a,\tilde e_b,\tilde e_c,\tilde e_d)
+\tilde T(\tilde J\tilde e_a, \tilde J\tilde e_b ,\tilde J\tilde e_c ,\tilde J\tilde e_d )\\
&+&\tilde T(\tilde J\tilde e_a, \tilde J\tilde e_b,\tilde e_c,\tilde e_d)
+\tilde T(\tilde J\tilde e_a, \tilde e_b,\tilde J\tilde e_c,\tilde e_d)\\
&+&\tilde T(\tilde J\tilde e_a, \tilde e_b,\tilde e_c,\tilde J\tilde e_d )
+\tilde T(\tilde e_a,\tilde J\tilde e_b ,\tilde J\tilde e_c, \tilde e_d)\\
&+&\tilde T(\tilde e_a,\tilde J\tilde e_b,\tilde e_c,\tilde J\tilde e_d)
+\tilde T(\tilde e_a,\tilde e_b,\tilde J\tilde e_c,\tilde J\tilde e_d)\,.
\end{eqnarray*}

We establish one implication of Theorem \ref{thm-1.2} (2) by showing:

\begin{theorem}\label{thm-2.1}
If $\tilde{\mathcal{M}}=(\tilde M,\tilde g,\tilde{\mathcal{J}})$ is a para-Hermitian manifold, then
$\tilde{\mathcal{G}}(\tilde R)=0$.
\end{theorem}
\begin{proof} Introduce coordinates $(u_1,...,u_{2n})$ on $\tilde M$ so
$$\tilde{\mathcal{J}}\partial_{u_1}=\partial_{u_{n+1}},\ \dots\ { ,\ }
\tilde{\mathcal{J}}\partial_{u_n}=\partial_{u_{2n}},\
  \tilde{\mathcal{J}}\partial_{u_{n+1}}=\partial_{u_1},\ \dots\ { ,\ }
\tilde{\mathcal{J}}\partial_{u_{2n}}=\partial_{u_n}\,.$$
We shall let
indices $a,b,c, \dots$ range from $1$ to $2n$ and index the coordinate frame
$\{\xi_1,...,\xi_{2n}\}:=\{\partial_{u_1},...,\partial_{u_{2n}}\}$.
We also let indices $\alpha,\beta,\gamma, \dots$ range from $1$ to
$2n$. Let
$$\tilde g_{ab}:=\tilde g(\xi_a,\xi_b),\quad \tilde g_{\alpha\beta}:=\tilde g(\tilde J\xi_a,\tilde J\xi_b),\quad
  \tilde g_{a\beta}:=\tilde g(\xi_a,\tilde J\xi_b),\quad \tilde g_{\alpha b}:=\tilde g(\tilde J\xi_a,\xi_b)\,.$$
We have $\tilde g_{ab}=-\tilde g_{\alpha\beta}$ and $\tilde g_{a\beta}=-\tilde g_{\alpha b}$.
Let $\tilde g^{ab}$ be the inverse matrix. We adopt the {\it Einstein} convention and sum over repeated indices.
Let ``/'' denote ordinary partial
differentiation.
Let $\tilde\Gamma$ be the Christoffel symbols of the Levi-Civita connection.
We compute:
\begin{eqnarray*}
&&\tilde\Gamma_{abc}=\textstyle\frac12(\tilde g_{bc/a}+\tilde g_{ac/b}-\tilde g_{ab/c}),\qquad
  \tilde\Gamma_{ab}{}^d=\tilde g^{cd}\tilde\Gamma_{abc},\\
&&\tilde R_{abc}{}^d= \partial_{u_a}\tilde\Gamma_{bc}{}^d-
\partial_{u_b}\tilde\Gamma_{ac}{}^d+\tilde\Gamma_{ae}{}^d\tilde\Gamma_{bc}{}^e
   -\tilde\Gamma_{be}{}^d\tilde\Gamma_{ac}{}^e\,.
\end{eqnarray*}
This enables us to compute:
\medbreak\quad
$\tilde R_{abcd}=\tilde g_{df}\partial_{u_a}\tilde\Gamma_{bc}{}^f-\tilde g_{df}\partial_{u_b}\tilde\Gamma_{ac}{}^f+
    \tilde g_{df}\tilde\Gamma_{ae}{}^f\tilde\Gamma_{bc}{}^e
   -\tilde g_{df}\tilde\Gamma_{be}{}^f\tilde\Gamma_{ac}{}^e$
\smallbreak\qquad
$=\tilde\Gamma_{bcd/a}-\tilde g_{df/a}\tilde\Gamma_{bc}{}^f
-\tilde\Gamma_{acd/b}+\tilde g_{df/b}\tilde\Gamma_{ac}{}^f$
$+\tilde g^{el}\tilde\Gamma_{aed}\tilde\Gamma_{bcl}-\tilde g^{el}\tilde\Gamma_{bed}\tilde\Gamma_{acl}$
\smallbreak\qquad
$=\tilde\Gamma_{bcd/a}-\tilde g^{fl}\tilde g_{df/a}\tilde\Gamma_{bcl}-\tilde\Gamma_{acd/b}$
$+\tilde g^{fl}\tilde g_{df/b}\tilde \Gamma_{acl}$
$+\tilde g^{el}\tilde\Gamma_{aed}\tilde\Gamma_{bcl}-\tilde g^{el}\tilde\Gamma_{bed}\tilde\Gamma_{acl}$.
\medbreak\noindent
We first study the linear terms in the second
derivatives of the metric:
$$\tilde\Gamma_{bcd/a}-\tilde\Gamma_{acd/b}=\textstyle\frac12\{\tilde g_{bd/ac}+\tilde g_{ac/bd}
           -\tilde g_{bc/ad}-\tilde g_{ad/bc}\}\,.$$
We examine the role $\tilde T_{abcd}^1:=\tilde g_{bd/ac}$ plays in the para-Gray identity; the remaining 3 terms play similar roles and
the argument is similar after permuting the indices appropriately. We use the fact that $\tilde J^*\tilde g=-\tilde g$ and apply $\tilde{\mathcal{G}}$ to compute
\begin{eqnarray*}
\tilde{\mathcal{G}}(\tilde T^1)_{abcd}&=&\tilde g_{bd/ac}+\tilde g_{\beta \delta/\alpha \gamma}+\tilde g_{\beta d/\alpha c}+\tilde g_{bd/\alpha \gamma}\\
&+&\tilde g_{b\delta/\alpha c}+\tilde g_{\beta d/a\gamma}+\tilde g_{\beta \delta/ac}+\tilde g_{b\delta/a\gamma}\\
&=&\tilde g_{bd/ac}-\tilde g_{bd/\alpha \gamma}-\tilde g_{b\delta/\alpha c}+\tilde g_{bd/\alpha \gamma}
\\&+&\tilde g_{b\delta/\alpha c}-\tilde g_{b\delta/a\gamma}-\tilde g_{bd/ac}+\tilde g_{b\delta/a\gamma}\\
&=&0\,.
\end{eqnarray*}
Next we examine the terms which are quadratic in the first derivatives of the metric; there are three different kinds of terms which must be symmetrized:
$$\tilde T^2_{abcd}:=\tilde g^{fe}\tilde g_{ad/f}\tilde g_{bc/e},\quad
  \tilde T^3_{abcd}:=\tilde g^{fe}\tilde g_{af/d}\tilde g_{bc/e},\quad
  \tilde T^4_{abcd}:=\tilde g^{fe}\tilde g_{af/d}\tilde g_{be/c}\,.$$
The remaining quadratic terms arise by permuting the roles of $\{a,b,c,d\}$ in these expressions. We compute:
\medbreak\qquad
$\tilde{\mathcal{G}}(\tilde T^2)_{abcd}=\tilde g^{fe}\{\tilde g_{ad/f}\tilde g_{bc/e}+\tilde g_{\alpha \delta /f}\tilde g_{\beta \gamma /e}+
\tilde g_{\alpha d/f}\tilde g_{\beta c/e}+\tilde g_{\alpha d/f}\tilde g_{b\gamma /e}$
\smallbreak\qquad\qquad\qquad\qquad\phantom{.}$+\tilde g_{\alpha \delta /f}\tilde g_{bc/e}+\tilde g_{ad/f}\tilde g_{\beta \gamma/e}+
\tilde g_{a\delta /f}\tilde g_{\beta c/e}+\tilde g_{a\delta /f}\tilde g_{b\gamma /e}\}$
\smallbreak\qquad\qquad\qquad\phantom{.}
$=\tilde g^{fe}\{\tilde g_{ad/f}\tilde g_{bc/e}+\tilde g_{ad/f}\tilde g_{bc/e}+
\tilde g_{a\delta/f}\tilde g_{b\gamma/e}-\tilde g_{a\delta/f}\tilde g_{b\gamma /e}$
\smallbreak\qquad\qquad\qquad\qquad\phantom{a}
$-\tilde g_{ad/f}\tilde g_{bc/e}-\tilde g_{ad/f}\tilde g_{bc/e}-\tilde g_{a\delta /f}\tilde g_{b\gamma/e}+\tilde g_{a\delta /f}\tilde g_{b\gamma /e}\}=0$,
\medbreak\qquad
$\tilde{\mathcal{G}}(\tilde T^3)_{abcd}=\tilde g^{fe}\{\tilde g_{af/d}\tilde g_{bc/e}+\tilde g_{\alpha f/\delta }\tilde g_{\beta \gamma /e}
+\tilde g_{\alpha f/d}\tilde g_{\beta c/e}+\tilde g_{\alpha f/d}\tilde g_{b\gamma /e}$
\smallbreak\qquad\qquad\qquad\qquad\phantom{.}
$+\tilde g_{\alpha f/\delta }\tilde g_{bc/e}+\tilde g_{af/d}\tilde g_{\beta \gamma /e}
+\tilde g_{af/\delta}\tilde g_{\beta c/e}+\tilde g_{af/\delta }\tilde g_{b\gamma /e}\}$
\smallbreak\qquad\qquad\qquad\phantom{.}
$=\tilde g^{fe}\{\tilde g_{af/d}\tilde g_{bc/e}-\tilde g_{\alpha f/\delta }\tilde g_{bc /e}
-\tilde g_{\alpha f/d}\tilde g_{b\gamma/e}+\tilde g_{\alpha f/d}\tilde g_{b\gamma /e}$
\smallbreak\qquad\qquad\qquad\qquad\phantom{.}
$+\tilde g_{\alpha f/\delta }\tilde g_{bc/e}-\tilde g_{af/d}\tilde g_{bc/e}
-\tilde g_{af/\delta}\tilde g_{b\gamma/e}+\tilde g_{af/\delta }\tilde g_{b\gamma /e}\}=0.$
\medbreak\noindent The final term requires a bit more work.
\medbreak\qquad
$\tilde{\mathcal{G}}(\tilde T^4)_{abcd}
=\tilde g^{fe}\{\tilde g_{af/d}\tilde g_{be/c}+\tilde g_{\alpha f/\delta }\tilde g_{\beta e/\gamma }
+\tilde g_{\alpha f/d}\tilde g_{\beta e/c}+\tilde g_{\alpha f/d}\tilde g_{be/\gamma }
$\smallbreak\qquad\qquad\qquad\qquad\phantom{.}$
+\tilde g_{\alpha f/\delta}\tilde g_{be/c}+\tilde g_{af/d}\tilde g_{\beta e/\gamma }
+\tilde g_{af/\delta }\tilde g_{\beta e/c}
+\tilde g_{af/\delta }\tilde g_{be/\gamma }\}$,
\medbreak\qquad
$\tilde g^{fe}\tilde g_{af/d}\tilde g_{be/c}+\tilde g^{fe}\tilde g_{\alpha f/d}\tilde g_{\beta e/c}=
 \tilde g^{fe}\tilde g_{af/d}\tilde g_{be/c}-\tilde g^{\theta\varepsilon}\tilde g_{a\theta/d}\tilde
g_{b\varepsilon/c}=0$,
\medbreak\qquad
$\tilde g^{fe}\tilde g_{\alpha f/\delta}\tilde g_{\beta e/\gamma}+\tilde g^{fe}\tilde g_{af/\delta}\tilde g_{be/\gamma}
=-\tilde g^{\theta\varepsilon}\tilde g_{a\theta/\delta}\tilde g_{b\varepsilon/\gamma}+\tilde g^{fe}\tilde
g_{af/\delta}\tilde g_{be/\gamma}=0$,
\medbreak\qquad
$\tilde g^{fe}\tilde g_{\alpha f/d}\tilde g_{be/\gamma}+\tilde g^{fe}\tilde g_{af/d}\tilde g_{\beta e/\gamma}
=-\tilde g^{\theta\varepsilon}\tilde g_{a\theta/d}\tilde g_{\beta\varepsilon/\gamma}+\tilde g^{fe}\tilde g_{af/d}\tilde
g_{\beta e/\gamma}=0$,
\medbreak\qquad
$\tilde g^{fe}\tilde g_{\alpha f/\delta }\tilde g_{be/c}+\tilde g^{fe}\tilde g_{af/\delta}\tilde g_{\beta e/c}
=-\tilde g^{\theta\varepsilon}\tilde g_{a\theta/\delta}\tilde g_{\beta\varepsilon/c}+\tilde g^{fe}\tilde g_{af/\delta}\tilde
g_{\beta e/c}=0$.
\medbreak\noindent
This establishes the para-Gray identity for para-Hermitian manifolds.
\end{proof}

\section{The Tricerri-Vanhecke curvature decomposition}\label{sect-3}

\subsection{Hermitian models}
Let $(V,\langle\cdot,\cdot\rangle,J)$ be a Hermitian structure. Extend
the inner product to tensors of all types.  Let $\Omega(x,y):=\langle x,Jy\rangle$ be the Kaehler form and
let $\mathfrak{A}(V)\subset\otimes^4V^*$ be the space of algebraic curvature tensors on $V$.
Set
\medbreak\qquad\qquad
$S_{0,+}^2(V^*,J):=\{\theta\in S^2(V^*):J^*\theta=\theta,\theta\perp\langle\cdot,\cdot\rangle\}$,
\smallbreak\qquad\qquad
$\Lambda_{0,+}^2(V^*,J):=\{\theta\in\Lambda^2(V^*):J^*\theta=\theta,\theta\perp\Omega\}$,
\smallbreak\qquad\qquad
$S_-^2(V^*,J):=\{\theta\in S^2(V^*):J^*\theta=-\theta\}$,
\smallbreak\qquad\qquad
$\Lambda_-^2(V^*,J):=\{\theta\in\Lambda^2(V^*):J^*\theta=-\theta\}$,
\smallbreak\qquad\qquad
$\mathcal{U}:=\{U\in\operatorname{GL}_{\mathbb{R}}(V):UJ=JU\quad\text{and}\quad U^*\langle\cdot,\cdot\rangle=\langle\cdot,\cdot\rangle\}$.
\medbreak

Pullback defines a natural orthogonal action of the unitary group $\mathcal{U}$,
by the orthogonal group $O(V,\langle\cdot,\cdot\rangle)$, and by
the general linear group $\operatorname{GL}_{\mathbb{R}}(V)$  on $V^*\otimes V^*$ and on $\mathfrak{A}(V)$. As a
$\operatorname{GL}_{\mathbb{R}}(V)$
 module, there is a direct sum decomposition of
$$V^*\otimes V^*=S^2(V^*)\oplus\Lambda^2(V^*)$$
into the symmetric and the anti-symmetric $2$-tensors,
respectively; these modules are irreducible ${\operatorname{GL}_{\mathbb{R}}}(V)$ modules. ${\Lambda^2}(V^*)$ is an irreducible
$O(V,\langle\cdot,\cdot\rangle)$ module. Let
{$S_0^2(V^\ast,\langle\cdot,\cdot\rangle)$} be the subspace of trace free symmetric 
{$2$-tensors}. There is
a further irreducible orthogonal decomposition of
$$
S^2(V^*)=\langle\cdot,\cdot\rangle\cdot
\mathbb{R}\oplus S_0^2({V^*},\langle\cdot,\cdot\rangle)\,.$$
Finally,
as $\mathcal{U}$ modules, we have an orthogonal direct sum decomposition:
\begin{equation}\begin{array}{rrrrrr}\label{eqn-3.a}
V^*\otimes V^*&=\langle\cdot,\cdot\rangle\cdot\mathbb{R}&\oplus
& S_{0,+}^2(V^*,J)&\oplus &S_-^2(V^*,J)\phantom{\,.}\\
&\oplus\phantom{..A}\Omega\cdot\mathbb{R}&\oplus&\Lambda_{0,+}^2(V^*,J)
&\oplus&\Lambda_-^2(V^*,J)\,.\vphantom{\vrule height 11pt}
\end{array}\end{equation}
If $\theta\in V^*\otimes V^*$, let $\theta_{0,+,S}$, $\theta_{-,S}$, and $\theta_{-,\Lambda}$ denote the components of $\theta$
in $S_{0,+}^2(V^*,J)$, $S_-^2(V^*,J)$, and $\Lambda_-^2(V^*,J)$, respectively.

Let $\{e_i\}$ be a basis for $V$. Let $\varepsilon_{ij}:=\langle e_i,e_j\rangle$ 
and let $\varepsilon^{ij}$ be the inverse matrix. {Let $A\in \mathfrak{A}(V)$.
 Let $\tau$, $\tau^\star$,
$\rho$, and
$\rho^\star$ be the scalar curvature, the
$\star$-scalar curvature, the Ricci tensor, and the
$\star$-Ricci tensor:
\begin{equation}\label{eqn-3.b}
\begin{array}{ll}
\rho(x,y):=\varepsilon^{ij}A(e_i,x,y,e_j),&\tau:=\varepsilon^{ij}\rho(e_i,e_j),\\
\rho^\star(x,y):=\varepsilon^{ij}A(e_i,x,Jy,Je_j),&\tau^\star:=\varepsilon^{ij}\rho^\star(e_i,e_j)\,.\vphantom{\vrule height 11pt}
\end{array}\end{equation}

We refer to \cite{TV81} for the proof of the following result:

\begin{theorem}\label{thm-3.1}
Let $(V,\langle\cdot,\cdot\rangle,J)$ be a Hermitian structure. \begin{enumerate}
\item We have the following orthogonal direct sum decomposition of $\mathfrak{A}(V)$ into irreducible $\mathcal{U}$ modules:
\begin{enumerate}
\item If $2n=4$,
$\mathfrak{A}(V)=\WW_1\oplus\WW_2\oplus\WW_3\oplus\WW_4\oplus\WW_7
\oplus\WW_8\oplus\WW_9$.
\item If $2n=6$,
$\mathfrak{A}(V)=\WW_1\oplus\WW_2\oplus\WW_3\oplus\WW_4\oplus\WW_5\oplus\WW_7
\oplus\WW_8\oplus\WW_9\oplus\WW_{10}$.
\item If $2n\ge8$,
$\mathfrak{A}(V)=\WW_1\oplus\WW_2\oplus\WW_3\oplus\WW_4\oplus\WW_5\oplus\WW_6\oplus\WW_7
\oplus\WW_8\oplus\WW_9\oplus\WW_{10}$.
\end{enumerate}
We have $\WW_1\approx\WW_4$ and, if $2n\ge6$, $\WW_2\approx\WW_5$. The other $\mathcal{U}$ modules
appear with multiplicity 1.
\item We have that:
\begin{enumerate}
\item  $\tau\oplus\tau^\star:\WW_1\oplus\WW_4\mapright{\approx}\mathbb{R}\oplus\mathbb{R}$.
\item If $2n=4$, $\rho_{0,+,S}:\WW_2\mapright{\approx} S_{0,+}^2(V^*,J)$.
\item If $2n\ge6$, $\rho_{0,+,S}\oplus\rho_{0,+,S}^\star:\WW_2\oplus\WW_5\mapright{\approx} S_{0,+}^2(V^*,J)\oplus S_{0,+}^2(V^*,J)$.
\item $\WW_3
=\{A\in\mathfrak{A}(V):A(x,y,z,w)=A(Jx,Jy,z,w)\ \forall\ x,y,z,w\}\cap\ker(\rho)$.
\item If $2n\ge8$, $\WW_6=\ker(\rho\oplus\rho^\star)\cap\{A\in\mathfrak{A}(V):J^*A=A\}\cap\WW_3^\perp$.
\item $\WW_7=\{A\in\mathfrak{A}(V):A(Jx,y,z,w)=A(x,y,Jz,w)\ \forall\  x,y,z,w\}$.
\item $\rho_{-,S}:\WW_8\mapright{\approx} S_-^2(V^*,J)$.
\item $\rho_{-,\Lambda}^\star :\WW_9\mapright{\approx}\Lambda_-^2(V^*,J)$.
\item If $2n\ge6$, $\WW_{10}=\{A\in\mathfrak{A}(V):J^*A=-A\}\cap\ker(\rho\oplus\rho^\star)$.
\end{enumerate}
\end{enumerate}
\end{theorem}

\subsection{Para-Hermitian models} Let $(\tilde V,\widetilde{\langle\cdot,\cdot\rangle},\tilde J)$
 be a para-Hermitian structure; the metric is non-degenerate on the space of algebraic curvature tensors
$\mathfrak{A}(\tilde V)$. Let $\tilde\Omega(x,y):=\widetilde{\langle x,\tilde Jy\rangle}$ be the para-Kaehler  form. We have
$$\tilde J^*\tilde\Omega=-\tilde\Omega\quad\text{and}\quad
\tilde J^*\widetilde{\langle\cdot,\cdot\rangle}=-\widetilde{\langle\cdot,\cdot\rangle}\,.$$
Set\medbreak\qquad\qquad
$S_+^2(\tilde V^*,\tilde J):=\{\theta\in S^2(\tilde V^*):\tilde J^*\theta=\theta\}$,
\smallbreak\qquad\qquad
$\Lambda_+^2(\tilde V^*,\tilde J):=\{\theta\in\Lambda^2(\tilde V^*):\tilde J^*\theta=\theta\}$,
\smallbreak\qquad\qquad
$S_{0,-}^2(\tilde V^*,\tilde J):=\{\theta\in S^2(\tilde V^*):\tilde J^*\theta=-\theta,\theta\perp\widetilde{\langle\cdot,\cdot\rangle}\}$,
\smallbreak\qquad\qquad
$\Lambda_{0,-}^2(\tilde V^*,\tilde J):=\{\theta\in\Lambda^2(\tilde V^*):\tilde J^*\theta=-\theta,\theta\perp\tilde\Omega\}$,
\smallbreak\qquad\qquad
$\tilde{\mathcal{U}}:=\{\tilde U\in\operatorname{GL}_{\mathbb{R}}(\tilde V):\tilde U\tilde J=\tilde J\tilde U\quad\text{and}\quad
\tilde U^*\widetilde{\langle\cdot,\cdot\rangle}=\widetilde{\langle\cdot,\cdot\rangle}\}$.
\medbreak\noindent
Fix a
basis $\{\tilde e_i\}$ for $\tilde V$ and let $\tilde\varepsilon_{ij}$ be the components of the inner product relative to this basis. If
$\tilde A$ is an algebraic curvature tensor, define:
$$
\begin{array}{ll}
\rho(x,y):={\tilde\varepsilon}^{ij}\tilde A({\tilde e_i},x,y,{\tilde e_j}),&\tau:={ \tilde\varepsilon}^{ij}\rho({
\tilde e_i},{\tilde e_j}),\\
\rho^\star(x,y):=-{\tilde\varepsilon}^{ij}
\tilde A({\tilde e_i},x,\tilde Jy,\tilde J{\tilde e_j}),&\tau^\star:={\tilde\varepsilon}^{ij}\rho^\star({\tilde
e_i},{\tilde e_j})\,.
\vphantom{\vrule height 11pt}\end{array}$$
 The decomposition of
Equation (\ref{eqn-3.a}) extends to this setting to become:
$$\begin{array}{rrrrrrr}
\tilde V^*\otimes\tilde V^*&=&\widetilde{\langle\cdot,\cdot\rangle}\cdot\mathbb{R}&\oplus& S_{0,-}^2(\tilde V^*,\tilde J)&\oplus& S_+^2(\tilde V^*,\tilde
J)\phantom{\,.}\\
  &\oplus&\tilde\Omega\cdot\mathbb{R}&\oplus&\Lambda_{0,-}^2(\tilde V^*,\tilde J)&\oplus&\Lambda_+^2(\tilde V^*,\tilde J)\,.
\vphantom{\vrule height 11pt}\end{array}$$
\begin{theorem}\label{thm-3.2}
Let $(\tilde V,\widetilde{\langle\cdot,\cdot\rangle},\tilde J)$ be a para-Hermitian structure. \begin{enumerate}
\item We have the following orthogonal direct sum decomposition of $\mathfrak{A}(\tilde V)$ into irreducible $\tilde{\mathcal{U}}$ modules:
\begin{enumerate}
\item If $2n=4$,
$\mathfrak{A}(\tilde V)=\WWW_1\oplus\WWW_2\oplus\WWW_3\oplus\WWW_4\oplus\WWW_7
\oplus\WWW_8\oplus\WWW_9$.
\item If $2n=6$,
$\mathfrak{A}(\tilde V)=\WWW_1\oplus\WWW_2\oplus\WWW_3\oplus\WWW_4\oplus\WWW_5\oplus\WWW_7
\oplus\WWW_8\oplus\WWW_9\oplus\WWW_{10}$.
\item If $2n\ge8$,
$\mathfrak{A}(\tilde V)=\WWW_1\oplus\WWW_2\oplus\WWW_3\oplus\WWW_4\oplus\WWW_5\oplus\WWW_6\oplus\WWW_7
\oplus\WWW_8\oplus\WWW_9\oplus\WWW_{10}$.
\end{enumerate}
We have $\WWW_1\approx\WWW_4$ and, if $2n\ge6$, $\WWW_2\approx\WWW_5$. The other $\tilde{\mathcal{U}}$ modules
appear with multiplicity 1.
\item We have that:
\begin{enumerate}
\item  $\tau\oplus\tau^\star:\WWW_1\oplus\WWW_4\mapright{\approx}\mathbb{R}\oplus\mathbb{R}$.
\item If $2n=4$, $\rho_{0,-,S}:\WWW_2\mapright{\approx}S_{0,-}^2(\tilde V^*,\tilde{J})$.
\item If $2n\ge6$, $\rho_{0,-,S}\oplus\rho_{0,-,S}^\star:\WWW_2\oplus\WWW_5\mapright{\approx}S_{0,-}^2(\tilde
V^*,\tilde{J})\oplus S_{0,-}^2(\tilde V^* ,\tilde{J})$.
\item $\WWW_3
=\{\tilde A\in\mathfrak{A}(\tilde V):\tilde A(x,y,z,w)=-\tilde A(\tilde Jx,\tilde Jy,z,w)\ \forall\
x,y,z,w\}$\newline$\phantom{AA}\cap\ker(\rho)$.
\item If $2n\ge8$, $\WWW_6=\ker(\rho\oplus\rho^\star)\cap\{\tilde A\in\mathfrak{A}(\tilde{V}):\tilde J^*\tilde A=\tilde
A\}\cap\WWW_3^\perp$.
\item $\WWW_7=\{\tilde A\in\mathfrak{A}(\tilde V):\tilde A(\tilde Jx,y,z,w)=\tilde A(x,y,\tilde Jz,w)\;\ \forall\
x,y,z,w\}$.
\item $\rho_{+,S}:\WWW_8\mapright{\approx}S_+^2(\tilde V^*,\tilde J)$.
\item $\rho_{+,\Lambda}^\star:\WWW_9\mapright{\approx}\Lambda_+^2(\tilde V^*,\tilde J)$.
\item If $2n\ge6$, $\WWW_{10}=\{\tilde A\in\mathfrak{A}(\tilde V):\tilde J^*\tilde A=-\tilde A\}\cap\ker(\rho\oplus\rho^\star)$.
\end{enumerate}
\end{enumerate}
\end{theorem}

\begin{proof} Let $(V,\langle\cdot,\cdot\rangle,J)$ be a Hermitian structure. We let
$V_{\mathbb{C}}:=V\otimes_{\mathbb{R}}\mathbb{C}$ be the complexification of $V$. We extend $\langle\cdot,\cdot\rangle$ to be complex bi-linear and we extend $J$
to be complex linear. We extend an element of $\mathfrak{A}(V)$ to be complex linear to define
$$
\mathfrak{A}(V_{\mathbb{C}})
  :=\{A_{\mathbb{C}}\in\otimes^4V_{\mathbb{C}}^*:\text{Equation (\ref{eqn-1.a})
holds}\}=\mathfrak{A}(V)\otimes_{\mathbb{R}}\mathbb{C}\,.
$$
Let $A\in\mathfrak{A}(V)$. If $\{\xi_i\}$ is any {$\mathbb{C}$-basis}
 for $V_{\mathbb{C}}$, then Equation
(\ref{eqn-3.b}) remains valid where $\varepsilon_{ij}:=\langle
\xi_i,\xi_j\rangle$. Let
$$
\mathcal{U}_{\mathbb{C}}:=\{U\in\operatorname{GL}_{\mathbb{C}}(V_{\mathbb{C}}):JU=UJ\text{ and }U^*\langle\cdot,\cdot\rangle=\langle\cdot,\cdot\rangle\}\,.
$$

Let $\{e_1,...,e_n,f_1,...,f_n\}$ be an orthonormal basis for $V$ where
$$Je_i=f_i\quad\text{and}\quad Jf_i=-e_i\,.$$
We let
\begin{eqnarray*}
&&\tilde e_i:=\sqrt{-1}e_i,\quad\tilde f_i:=-f_i,\quad\tilde J:=\sqrt{-1}J,\quad
  \tilde V:=\operatorname{Span}_{\mathbb{R}}\{\tilde e_i,\tilde f_i\},\\
&&\tilde{\mathcal{U}}:=\{\tilde U\in\operatorname{GL}_{\mathbb{R}}(\tilde V):\tilde U\tilde J=\tilde J\tilde U\quad\text{and}\quad \tilde
U^*\langle\cdot,\cdot\rangle=\langle\cdot,\cdot\rangle\} ={\mathcal{U}}_{\mathbb{C}}\cap\operatorname{GL}_{\mathbb{R}}(\tilde
V),\\ &&\mathfrak{A}(\tilde V)=\mathfrak{A}(V_{\mathbb{C}})\cap\otimes^4\tilde V^*\,.
\end{eqnarray*}

Since $V$ has a positive definite metric, $\tilde V$ inherits a metric $\widetilde{\langle\cdot,\cdot\rangle}$ of neutral signature $(n,n)$; the vectors $\tilde
e_i$ being timelike and the vectors $\tilde f_i$ being spacelike. Certain sign changes now manifest themselves:
$$\rho^\star{(x,y)}=-\varepsilon^{ij}A(e_i,x,\tilde Jy,\tilde Je_i),\quad\tau^\star=\varepsilon^{ij}
\rho^\star(e_i,e_j)\,.$$
In the decomposition of Equation (\ref{eqn-3.a}), we have
\begin{eqnarray*}
&&S^2_\pm(V^*,J)\otimes_{\mathbb{R}}{\mathbb{C}}=\{\theta\in\otimes^2V^*_{\mathbb{C}}:\theta(x,y)=\theta(y,x),\theta(Jx,Jy)=\pm\theta(x,y)\}\\
&=&\{\theta\in\otimes^2\tilde V^*_{\mathbb{C}}:\theta(x,y)=\theta(y,x),\theta(\tilde Jx,\tilde
Jy)=\mp\theta(x,y)\}=S^2_{\mp}(\tilde V^*,\tilde J)\otimes_{\mathbb{R}}{\mathbb{C}},\\
&&\Lambda^2_{\pm}(V^*,J)\otimes_{\mathbb{R}}\mathbb{C}=\{\theta\in\otimes^2V^*_{\mathbb{C}}:\theta(x,y)=-\theta(y,x),\theta(Jx,Jy)=\pm\theta(x,y)\}\\
&=&\{\theta\in\otimes^2\tilde V^*_{\mathbb{C}}:\theta(x,y)=-\theta(y,x),\theta(\tilde Jx,\tilde Jy)=\mp\theta(x,y)\}
=\Lambda^2_{\mp}(\tilde V^*,\tilde
J)\otimes_{\mathbb{R}}\mathbb{C}\,.
\end{eqnarray*}
This defines a bijective correspondence which derives the decomposition of Theorem \ref{thm-3.2} from that of Theorem \ref{thm-3.1}. The correspondence is
reversible and hence the modules in Theorem \ref{thm-3.2} can not be decomposed further.
\end{proof}

\begin{remark}\rm We started in the Hermitian setting to deduce a theorem in the para-Hermitian setting.
Thus the Tricerri-Vanhecke decomposition works equally well in the pseudo-Hermitian setting by changing both the inner
product and the operator $J$. Suppose given integers $p$ and $q$ with $p+q=n$. By setting
\begin{eqnarray*}
&&
\tilde e_i:=\left\{\begin{array}{rrr}\sqrt{-1}e_i&\text{if}&{1\le} i\le p\\
e_i&\text{if}&p<i\le n\end{array}\right\}\\
&&\tilde f_i:=\left\{\begin{array}{rrr}\sqrt{-1}f_i&\text{if}&{1\le} i\le p\\
f_i&\text{if}&p<i\le n\end{array}\right\}
\end{eqnarray*}
and by taking $\tilde J:=J$, we could create a pseudo-Hermitian model of signature $(2p,2q)$. The analogous correspondence
would then permit us to deduce a Tricerri-Vanhecke decomposition theorem in the pseudo-Hermitian signature as well.
\end{remark}

\section{Linearizing the problem}\label{sect-4}
We fix a para-Hermitian structure $(\tilde V,\widetilde{\langle\cdot,\cdot\rangle},\tilde J)$ hence forth.
If $\Theta\in\otimes^4\tilde V^*$, set
$$\mathcal{P}(\Theta)(x,y,z,w):=\Theta(x,z,y,w)+\Theta(y,w,x,z)-\Theta(x,w,y,z)-\Theta(y,z,x,w)\,.$$

\begin{lemma}\label{lem-4.1}If $\Theta\in S_-^2(\tilde V^*,\tilde J)\otimes S^2(\tilde V^*)$, then $\mathcal{P}(\Theta)$ is an algebraic curvature
tensor such that the complex model
$(\tilde V,\widetilde{\langle\cdot,\cdot\rangle},\tilde J,\mathcal{P}(\Theta))$
 is geometrically realizable by a
para-Hermitian manifold.
\end{lemma}

\begin{proof} Let $\{e_1,...,e_n,f_1,...,f_n\}$ be a basis for $\mathbb{R}^{2n}$. Define an inner product
$\Xi$ of signature $(n,n)$ on $\mathbb{R}^{2n}$
 whose non-zero entries are
$$\Xi(e_1,e_1)=...=\Xi(e_n,e_n)=-1\quad\text{and}\quad\Xi(f_1,f_1)=...=\Xi(f_n,f_n)=+1\,.$$
If  $v\in\mathbb{R}^{2n}$, expand
$v=x_1e_1+...+x_ne_n+y_1f_1+...+y_nf_n$ to define coordinates $(x_1,...,x_n,y_1,...,y_n)=(u_1,...,u_{2n})$. Define
$$
\tilde{\mathcal{J}}\partial_{x_1}:=\partial_{y_1},\quad...\quad,
\tilde{\mathcal{J}}\partial_{x_n}:=\partial_{y_n},\quad
\tilde{\mathcal{J}}\partial_{y_1}:=\partial_{x_1},\quad...\quad,
\tilde{\mathcal{J}}\partial_{y_n}:=\partial_{x_n}\,.
$$
Let $\Theta\in {S_-^2(\tilde V^*,\tilde J)}\otimes S^2(\tilde V^*)$.
Define
\begin{equation}\label{eqn-2.d}
\tilde g_{ij}:=\Xi_{ij}+2\Theta_{ijkl}u^ku^l\,.
\end{equation}
Since $\Theta(x,y,z,w)=-\Theta(\tilde Jx,\tilde Jy,z,w)$,
$\tilde{\mathcal{J}}^*{\tilde{g}}=-{\tilde{g}}$.
 Let $B_\epsilon$ be the Euclidean ball of
radius $\epsilon>0$ centered at the origin. Since
$\tilde g$ is non-singular at the origin, there exists $\epsilon>0$ so $\tilde g$ is non singular on $ B_\epsilon$;
let
$\tilde{\mathcal{M}}:=(B_{ \epsilon},\tilde g,\tilde{\mathcal{J}})$
  be the resulting para-Hermitian manifold. Since the first derivatives of the metric vanish at $0$,
\medbreak\quad
$R(\partial_{u_i},\partial_{u_j},\partial_{u_k},\partial_{u_l})(0)
=\textstyle\frac12\{\partial_{u_i}\partial_{u_k}\tilde g_{jl}+\partial_{u_j}\partial_{u_l}\tilde g_{ik}
-\partial_{u_i}\partial_{u_l}\tilde g_{jk}-\partial_{u_j}\partial_{u_k}\tilde g_{il}\}$
\medbreak\qquad\quad
$=\Theta_{ikjl}+\Theta_{jlik}-\Theta_{iljk}-\Theta_{jkil}=\mathcal{P}(\Theta)$.
\end{proof}

\section{The proof of Theorem \ref{thm-1.2} (2)}\label{sect-5}
Let $\WWW_G$ be the space of algebraic curvature tensors such that the para-Gray identity holds. Let
$$\mathfrak{P}:=\mathcal{P}\{S^2_-(\tilde V^*,\tilde J)\otimes S^2({\tilde V}^*)\}\,.$$
$\mathfrak{P}$ and $\WWW_G$ are linear subspaces of $\mathfrak{A}(\tilde V)$ which are invariant under the action of the para-unitary group $\tilde{\mathcal{U}}$.
The results of Section
\ref{sect-3} reduce the proof of Theorem
\ref{thm-1.2} (2) to showing
 $\mathfrak{P}=\WWW_G$.
We begin our study with the following result:

\begin{lemma}\label{lem-5.1}
 $\mathfrak{P}\subset\WWW_G\subset\WWW_7^\perp$.
\end{lemma}

\begin{proof} By Lemma \ref{lem-4.1}, every element of $\mathfrak{P}$ can be geometrically realized by a para-Hermitian manifold. Theorem \ref{thm-2.1} now implies
$\mathfrak{P}\subset\WWW_G$.  We show $\WWW_G\subset\WWW_7^\perp$ by showing
$\WWW_G\cap\WWW_7=\{0\}$. Let $\tilde A\in\WWW_G\cap\WWW_7$. Since $\tilde A\in\WWW_7$, the curvature
symmetries imply additionally that
\medbreak\qquad
$\tilde A(\tilde Jx,y,z,w)=-\tilde A(\tilde Jx,y,w,z)=-\tilde A(x,y,\tilde Jw,z)=\tilde A(x,y,z,\tilde Jw)$
\medbreak\qquad\qquad$=-\tilde A(y,x,z,\tilde Jw)=-\tilde A(\tilde Jy,x,z,w)=\tilde A(x,\tilde Jy,z,w)$.
\medbreak\noindent
Since $\tilde A\in\WWW_G$, we have
\medbreak\qquad
$0=\tilde A(x,y,z,w)+\tilde A(\tilde Jx,\tilde Jy,\tilde Jz,\tilde Jw)$
\medbreak\qquad\quad
$+\tilde A(\tilde Jx,\tilde Jy,z,w)+\tilde A(x,y,\tilde Jz,\tilde Jw)+\tilde A(\tilde Jx,y,\tilde Jz,w)$
\medbreak\qquad\quad
$+\tilde A(x,\tilde Jy,z,\tilde Jw)+\tilde A(\tilde Jx,y,z,\tilde Jw)+\tilde A(x,\tilde Jy,\tilde Jz,w)$
\medbreak\qquad\quad
$=8\tilde A(x,y,z,w)$.\end{proof}

We continue our study with:

\begin{lemma}\label{lem-5.2}
\ \begin{enumerate}
\item $\tau\oplus\tau^\star:\mathfrak{P}\rightarrow\mathbb{R}\oplus\mathbb{R}\rightarrow0$. Thus
$\WWW_1\oplus\WWW_4\subset\mathfrak{P}$.
\item If $2n=4$, then $\rho_{0,-,S}:\mathfrak{P}\rightarrow S_{0,-}^2(\tilde V^*,\tilde J)\rightarrow0$. Thus
$\WWW_2\subset\mathfrak{P}$.
\item $\rho_{+,S}:\mathfrak{P}\rightarrow S_+^2(\tilde V^*,\tilde J)\rightarrow0$. Thus $\WWW_8\subset\mathfrak{P}$.
\item ${\rho_{+,\Lambda}^\star}:\mathfrak{P}\rightarrow\Lambda^2_+(\tilde V^*,\tilde J)\rightarrow0$.
Thus
$\WWW_9\subset\mathfrak{P}$.
\item If $2n\ge6$, then  $\{\rho_{0,-,S}\oplus\rho_{0,-,S}^\star\}:\mathfrak{P}\rightarrow\{S_{0,-}^2(\tilde
V^*,\tilde J)\oplus S_{0,-}^2(\tilde V^*,\tilde J)\}\rightarrow0$. Thus
$\WWW_2\oplus\WWW_5\subset\mathfrak{P}$.
\item $\mathfrak{P}\cap\WWW_3\ne\{0\}$. Thus $\WWW_3\subset\mathfrak{P}$.
\item $\mathfrak{P}\cap\WWW_{10}\ne\{0\}$. Thus $\WWW_{10}\subset\mathfrak{P}$.
\item If $2n\ge6$, then $\mathfrak{P}\cap\WWW_6\ne\{0\}$. Thus $\WWW_6\subset\mathfrak{P}$.
\end{enumerate}\end{lemma}

\begin{proof}  As in the proof of Lemma \ref{lem-4.1}, we examine
metrics $\tilde g=\Xi+O(|u|^2)$; let $\tilde A\in\mathfrak{P}$ be the curvature tensor at the origin. Set $\tilde A^*(x,y,z,w):=\tilde A(x,y,\tilde Jz,\tilde
Jw)$. Let
$$\xi\circ\eta:=\textstyle\frac12(\xi\otimes\eta+\eta\otimes\xi)$$
denote the symmetric product. Let
$\varrho$ and
$\varepsilon$ be real constants. Consider the para-Hermitian metric:
$$\tilde g=\Xi-\varepsilon x_1^2(dx_1\circ dx_1-dy_1\circ dy_1)-\varrho x_1^2(dx_2\circ dx_2-dy_2\circ dy_2)\,.$$
The non-zero entries of $\tilde A$ are, up to the usual $\mathbb{Z}_2$ symmetries,
\begin{eqnarray*}
&&\tilde A(\partial_{x_1},\partial_{y_1},\partial_{y_1},\partial_{x_1})=-\varepsilon,\\
&&\tilde A(\partial_{x_1},\partial_{x_2},\partial_{x_2},\partial_{x_1})=\varrho,\quad
  \tilde A(\partial_{x_1},\partial_{y_2},\partial_{y_2},\partial_{x_1})=-\varrho\,.
\end{eqnarray*}
Since the $\{\partial_{x_i}\}$ are timelike and the $\{\partial_{y_i}\}$ are spacelike,
$\tau=2\varepsilon+4\varrho$ and {$\tau^\star=2\varepsilon$}
 so $\tau\oplus\tau^\star$ is a
surjective map from $\mathfrak{P}$ to $\mathbb{R}\oplus\mathbb{R}$. Thus Assertion (1) follows from Theorem \ref{thm-3.2}:
$$\WWW_1\oplus\WWW_4\subset\mathfrak{P}\,.$$

The non-zero entries in the Ricci tensor are given by:
$$\begin{array}{ll}
\rho(\partial_{x_1},\partial_{x_1})=-\varepsilon-2\varrho,&\rho(\partial_{y_1},\partial_{y_1})=\varepsilon,\\
  \rho(\partial_{x_2},\partial_{x_2})=-\varrho,&\rho(\partial_{y_2},\partial_{y_2})=\varrho\,.
\end{array}$$
We take $\varrho=-1$ and $\varepsilon=2$ to ensure $\rho$ is trace free and symmetric. We then have
$$\begin{array}{ll}
\rho_{+,S}(\partial_{x_1},\partial_{x_1})=1,&\rho_{0,-,S}(\partial_{x_1},\partial_{x_1})=-1,\\
\rho_{+,S}(\partial_{y_1},\partial_{y_1})=1,&\rho_{0,-,S}(\partial_{y_1},\partial_{y_1})=1,\\
\rho_{+,S}(\partial_{x_2},\partial_{x_2})=0,&\rho_{0,-,S}(\partial_{x_2},\partial_{x_2})=1,\\
\rho_{+,S}(\partial_{y_2},\partial_{y_2})=0,&\rho_{0,-,S}(\partial_{y_2},\partial_{y_2})=-1\,.
\end{array}$$
This shows that $\rho_{0,-,S}$ is non-zero on $\mathfrak{P}$; Assertion (2) now follows if $2n=4$ since $\WWW_5$ is not present:
$$\WWW_2\subset\mathfrak{P}\,.$$
It also shows $\rho_{+,S}$ is non-trivial on $\mathfrak{P}$ and establishes Assertion (3):
$$\WWW_8\subset\mathfrak{P}\,.$$

We clear the previous notation and consider:
$$\tilde g=\Xi-{4}\varepsilon x_1^2(-dx_1\circ dx_2+dy_1\circ dy_2)\,.$$
There is only one non-zero curvature entry
$\tilde A(\partial_{x_1},\partial_{y_1},\partial_{y_2},\partial_{x_1})=2\varepsilon$. We have:
$$\begin{array}{ll}
\tilde A^*(\partial_{x_1},\partial_{y_1},\partial_{x_2},\partial_{y_1})=2\varepsilon,&
\tilde A^*(\partial_{y_2},\partial_{x_1},\partial_{y_1},\partial_{x_1})=2\varepsilon,\\
\rho^\star(\partial_{x_1},\partial_{x_2})={2\varepsilon},&\rho^\star(\partial_{y_2},\partial_{y_1})={-2\varepsilon},\\
\rho^\star_\Lambda(\partial_{x_1},\partial_{x_2})=-\rho^\star_\Lambda(\partial_{x_2},\partial_{x_1})={\varepsilon},&
\rho^\star_\Lambda(\partial_{y_2},\partial_{y_1})=
-\rho^\star_\Lambda(\partial_{y_1},\partial_{y_2})={-\varepsilon},\\
 \rho_S^\star(\partial_{x_1},\partial_{x_2})=\rho_S^\star(\partial_{x_2},\partial_{x_1})={\varepsilon},&
\rho_S^\star(\partial_{y_1},\partial_{y_2})=\rho_S^\star(\partial_{y_2},\partial_{y_1})={-\varepsilon}\,.
\end{array}$$
This shows $\rho^\star_{+,\Lambda}=\rho^\star_\Lambda\ne0$ so $\tilde A$ has a non-trivial component in $\WWW_9$. This
completes the proof of Assertion (4):
$$\WWW_9\subset\mathfrak{P}\,.$$

Assume $2n\ge6$. We clear the
previous notation and consider:
\begin{eqnarray*}
\tilde g=\Xi-2\varrho x_1^2(-dx_1\circ dx_2+dy_1\circ dy_2)-2\varepsilon x_1^2(-dx_2\circ dx_3+dy_2\circ dy_3)\,.
\end{eqnarray*}
The non-zero curvatures now become:
\begin{eqnarray*}
&&\tilde A(\partial_{x_1},\partial_{y_1},\partial_{y_2},\partial_{x_1})=\varrho,\\
&&\tilde A(\partial_{x_1},\partial_{x_2},\partial_{x_3},\partial_{x_1})=-\varepsilon,\quad
\tilde A(\partial_{x_1},\partial_{y_2},\partial_{y_3},\partial_{x_1})=\varepsilon\,.
\end{eqnarray*}
Note that $\rho$ is always symmetric. We have
$$\begin{array}{ll}
\rho(\partial_{y_1},\partial_{y_2})=-\varrho,&\rho(\partial_{x_1},\partial_{x_2})=0,\\
\rho(\partial_{x_2},\partial_{x_3})=\varepsilon,&
\rho(\partial_{y_2},\partial_{y_3})=-\varepsilon\,.
\end{array}$$
This leads to the decomposition:
\medbreak\qquad
$\rho_{0,-,S}(\partial_{x_1},\partial_{x_2})=\phantom{a}\textstyle\frac12\varrho$,\quad
$\rho_{+,S}(\partial_{x_1},\partial_{x_2})=-\textstyle\frac12\varrho$,
\smallbreak\qquad
$\rho_{0,-,S}(\partial_{y_1},\partial_{y_2})=-\textstyle\frac12\varrho$,\quad
$\rho_{+,S}(\partial_{y_1},\partial_{y_2})=-\textstyle\frac12\varrho$,
\smallbreak\qquad
$\rho_{0,-,S}(\partial_{x_2},\partial_{x_3})=\phantom{-..}{\varepsilon}$,\quad
$\rho_{+,S}(\partial_{x_2},\partial_{x_3})=0$,
\smallbreak\qquad
$\rho_{0,-,S}(\partial_{y_2},\partial_{y_3})=\phantom{.}-\varepsilon$,\quad
$\rho_{+,S}(\partial_{y_2},\partial_{y_3})=0$.\medbreak\noindent
We have:
$$\begin{array}{ll}
\tilde A^*(\partial_{x_1},\partial_{y_1},\partial_{x_2},\partial_{y_1})=\varrho,&
\tilde A^*(\partial_{y_2},\partial_{x_1},\partial_{y_1},\partial_{x_1})=\varrho,\\
\tilde A^*(\partial_{x_1},\partial_{x_2},\partial_{y_3},\partial_{y_1})=-\varepsilon,&
\tilde A^*(\partial_{x_3},\partial_{x_1},\partial_{y_1},\partial_{y_2})=-\varepsilon,\\
\tilde A^*(\partial_{x_1},\partial_{y_2},\partial_{x_3},\partial_{y_1})=\varepsilon,&
\tilde A^*(\partial_{y_3},\partial_{x_1},\partial_{y_1},\partial_{x_2})=\varepsilon\,.
\end{array}$$
 Consequently $\rho^\star(\partial_{x_1},\partial_{x_2})={\varrho}$ and
$\rho^\star(\partial_{y_2},\partial_{y_1})={-\varrho}$. This yields:
\medbreak\qquad\qquad
$\rho_{0,-,S}^\star(\partial_{x_1},\partial_{x_2})=\phantom{-}{\textstyle\frac12\varrho}$,\quad
$\rho_{+,\Lambda}^\star(\partial_{x_1},\partial_{x_2})=\phantom{a}{\textstyle\frac12\varrho}$,
\smallbreak\qquad\qquad
$\rho_{0,-,S}^\star(\partial_{y_1},\partial_{y_2})={-\textstyle\frac12\varrho}$,\quad
$\rho_{+,\Lambda}^\star(\partial_{y_1},\partial_{y_2})={+\textstyle\frac12\varrho}$.
\medbreak\noindent
If we take $\varrho=0$ and $\varepsilon\ne0$, then $\rho_{0,-,S}\ne0$ and $\rho_{0,-,S}^\star=0$. Thus
\begin{eqnarray*}
&&\{S_{0,-}^2(\tilde V^*,\tilde J)\oplus 0\}\cap\{\rho_{0,-,S}\oplus\rho_{0,-,S}^\star\}\mathfrak{P}\ne\{0\}\quad\text{so}\\
&&\{S_{0,-}^2(\tilde V^*,\tilde J)\oplus 0\}\subset\{\rho_{0,-,S}\oplus\rho_{0,-,S}^\star\}\mathfrak{P}\,.
\end{eqnarray*}
On the other hand, if we take $\varrho\ne0$, then ${\rho_{0,-,S}^\star}\ne0$.
 Thus we have a non-zero component in the second
factor and
$$\{S_{0,-}^2(\tilde V^*,\tilde J)\oplus {S_{0,-}^2}(\tilde V^*,\tilde J)\}\subset\{\rho_{0,-,S}
\oplus\rho_{0,-,S}^\star\}\mathfrak{P}\,.$$
This establishes Assertion (5):
$$\WWW_2\oplus\WWW_5\subset\mathfrak{P}\,.$$

 To prove Assertion (6), we consider the metric
$$
\tilde g=\Xi-2\{x_1^2-y_1^2-x_2^2+y_2^2\}(-dx_1\circ dx_2+dy_1\circ dy_2)\,.
$$
The non-zero components of $\tilde A$ are then given, up to the usual $\mathbb{Z}_2$ symmetries by:
$$\begin{array}{ll}
\tilde A(\partial_{x_1},\partial_{y_1},\partial_{y_2},\partial_{x_1})=1,&\tilde A(\partial_{y_1},\partial_{x_1},\partial_{x_2},\partial_{y_1})=1,\\
\tilde A(\partial_{x_2},\partial_{y_1},\partial_{y_2},\partial_{x_2})=-1,&\tilde A(\partial_{y_2},\partial_{x_1},\partial_{x_2},\partial_{y_2})=-1\,.
\vphantom{\vrule height 11pt}\end{array}
$$
We have
$\rho=0$ and $\tilde A(\tilde Jx,\tilde Jy,z,w)=-\tilde A(x,y,z,w)$ for all $x$, $y$, $z$, and $w$. This shows $\tilde A\in\WWW_3$ and proves Assertion (6) by
showing
$$\WWW_3\subset\mathfrak{P}\,.$$

Let $2n\ge6$. We consider
$$
\tilde g=\Xi-2\{x_1^2+y_1^2\}(-dx_2\circ dx_3+dy_2\circ dy_3) \,.
$$
The non-zero curvatures are then
\begin{eqnarray*}
&&\tilde A(\partial_{x_1},\partial_{x_2},\partial_{x_3},\partial_{x_1})=-1,\quad \tilde A(\partial_{x_1},\partial_{y_2},\partial_{y_3},\partial_{x_1})=1,\\
&&\tilde A(\partial_{y_1},\partial_{x_2},\partial_{x_3},\partial_{y_1})=-1,\quad \tilde A(\partial_{y_1},\partial_{y_2},\partial_{y_3},\partial_{y_1})=1\,.
\end{eqnarray*}
We have $\rho=\rho^\star=0$. Since ${\tilde J}^*\tilde A=-\tilde A$,
 $\tilde A\in\WWW_{10}$; Assertion (7)
follows since
$$\WWW_{10}\subset\mathfrak{P}\,.$$

Let $2n\ge8$. We take
$$\tilde g=\Xi-{4}\{x_1x_2+y_1y_2\}(-dx_3\circ dx_4+dy_3\circ dy_4)\,.$$
The non-zero curvatures are
\begin{eqnarray*}
&&\tilde A(\partial_{x_1},\partial_{x_3},\partial_{x_4},\partial_{x_2})=
\tilde A(\partial_{y_1},\partial_{x_3},\partial_{x_4},\partial_{y_2})\\
&=&\tilde A(\partial_{x_1},\partial_{x_4},\partial_{x_3},\partial_{x_2})=
\tilde A(\partial_{y_1},\partial_{x_4},\partial_{x_3},\partial_{y_2})=-1,\\
&&\tilde A(\partial_{x_1},\partial_{y_3},\partial_{y_4},\partial_{x_2})=
\tilde A(\partial_{y_1},\partial_{y_3},\partial_{y_4},\partial_{y_2})\\
&=&\tilde A(\partial_{x_1},\partial_{y_4},\partial_{y_3},\partial_{x_2})=
\tilde A(\partial_{y_1},\partial_{y_4},\partial_{y_3},\partial_{y_2})=1\,.
\end{eqnarray*}
We observe that $\rho=\rho^\star=0$. Since {$\tilde A(\tilde Jx,\tilde Jy,z,w)\ne -\tilde A(x,y,z,w)$}, $\tilde
A\notin\WWW_3$. Thus $\tilde A$ has a non-zero component in
$\WWW_6\oplus\WWW_7$. As $\mathfrak{P}\perp\WWW_7$, $\tilde A$ has a non-zero component in $\WWW_6$ and
Assertion (8) follows; $\WWW_6\subset\mathfrak{P}$.
\end{proof}

\medbreak\noindent{\it Proof of Theorem \ref{thm-1.2} (2).} By Lemma \ref{lem-5.1}, we have
$\mathfrak{P}\subset\WWW_G\subset\WWW_7^\perp$. The assertion
$\WWW_7^\perp\subset\mathfrak{P}$ follows from the Tricerri-Vanhecke decomposition described in
Theorem \ref{thm-3.2} and from Lemma
\ref{lem-5.2}.\hfill\qedbox

\medbreak\noindent{\it Proof of Remark \ref{rmk-1.3}}. The construction given above yields
$\tilde{\mathcal{M}}$ with {$d\tilde\Omega_P=0$}
 realizing the given complex
curvature model $\tilde{\mathfrak{C}}$ at $P$. Imposing the para-Kaehler
identity $d\tilde\Omega\equiv0$ globally would imply that
$\tilde R\in\WWW_1\oplus\WWW_2\oplus\WWW_3$ so this is not possible in general. In
\cite{GBKNW08}, we considered a further variation
$$\tilde h:=\Xi+2\xi(dx_1\circ dx_1-dy_1\circ dy_1)+2\eta(dx_2\circ dx_2-dy_2\circ dy_2)$$
where $\{\xi,\eta\}$ are smooth functions vanishing to second order at $P$. We showed it was possible to choose $\{\xi,\eta\}$ so that the resulting metric had
constant scalar curvature and constant $\star$-scalar curvature. Since $\{\xi,\eta\}$ vanish to second order, $(\tilde M,\tilde h,\tilde{\mathcal{J}})$ realizes
$\tilde{\mathfrak{C}}$ at
$P$ as well and $d\tilde\Omega_{\xi,\eta}=0$. This establishes Remark
\ref{rmk-1.3}.
\hfill\qedbox

\section*{Acknowledgments} The research of all authors was partially supported by Project MTM2006-01432
(Spain). The research of P. Gilkey was also partially supported by  Project DGI SEJ2007-67810a (Spain) and research of S. Nik\v cevi\'c was also partially
supported by Project 144032 (Serbia).

\end{document}